\documentclass[10pt]{article}
\usepackage{amsfonts}
\usepackage{amsmath}
\usepackage{mathrsfs}
\usepackage{color}
\usepackage{mathrsfs,amscd,amssymb,amsthm,amsmath,bm,graphicx,psfrag,subfigure,url}

\makeatletter

\renewcommand{\@seccntformat}[1]{{\csname the#1\endcsname}{\normalsize .}\hspace{.5em}}
\makeatother

\usepackage{indentfirst}

\def \[{\begin{equation}}
\def \]{\end{equation}}

\newtheorem{thm}{Theorem}[section]

\newtheorem{claim}{Claim}
\newtheorem{lem}[thm]{Lemma}

\begin{document}

\setlength{\baselineskip}{15pt}
\begin{center}{\Large \bf The normalized Laplacian, degree-Kirchhoff index and spanning trees of graphs derived from the strong prism of linear polyomino chain}
\vspace{4mm}

{\large Xiaocong He$^a$\footnote{Corresponding author. \\
\hspace*{5mm}{\it Email addresses}: hexc2018@qq.com (X.C. He)}\vspace{2mm}}

$^a$School of Mathematics and Statistics,  Central South University, New campus,\\
Changsha, Hunan, 410083, P.R. China
\end{center}

%%%%%%%%%%%%%

\noindent {\bf Abstract}: Let $B_n$ be a linear polyomino chain with $n$ squares. Let $B_n^2$ be the graph obtained by the strong prism of a linear polyomino chain with $n$ squares, i.e. the strong product of $K_2$ and $B_n$. In this paper, explicit expressions for degree-Kirchhoff index and number of spanning trees of $B^2_n$ are determined, respectively. Furthermore, it is interesting to find that the degree-Kirchhoff index of $B^2_n$ is almost one eighth of its Gutman index.

\vspace{2mm} \noindent{\bf Keywords}: Linear polyomino chain; Normalized  Laplacian; Degree-Kirchhoff index; Spanning tree

\vspace{2mm}


\setcounter{section}{0}
\section{\normalsize Introduction}\setcounter{equation}{0}
In this paper, we only consider simple and undirected graphs. Let $G=(V(G), E(G))$ be a graph with the vertex set $V(G)=\{v_1,v_2,\ldots,v_n\}$ and the edge set $E(G)$. The \textit{adjacency matrix} of $G$ is a square matrix $A(G)=(a_{ij})_{n\times n}$ with entries $a_{ij}=1$ or $0$ according as the corresponding vertices $v_i$ and $v_j$ are adjacent or not. Let $D(G)={\rm diag}(d_1, d_2, \ldots, d_n)$ be the diagonal matrix of vertex degrees, where $d_i$ is the degree of $v_i$ in $G$ for
$1 \leqslant  i \leqslant n.$ The (\textit{combinatorial}) \textit{Laplacian matrix} of $G$ is defined as $L(G)=D(G)-A(G)$.

The classical distance between vertices $v_i$ and $v_j$ in a graph $G$, denoted by $d_{ij}$, is the length of a shortest path in $G$ connecting them. A well-known topological descriptor called \textit{Wiener index}, $W(G),$ was given by $W(G)=\sum_{i<j}d_{ij}$ in \cite{011}. Later, Gutman \cite{008} introduced the weighted version of Wiener index, namely \textit{Gutman index} of $G$, which was defined as $Gut(G)=\sum_{i<j}d_id_jd_{ij}$. In \cite{008}, it was shown that when $G$ is a tree on $n$ vertices, then the Wiener index and Gutman index are closely related by $Gut(G)=4W(G)-(2n-1)(n-1).$

On the basis of electrical network theory, Klein and Randi\'{c} \cite{1} proposed a novel distance function, namely the \textit{resistance distance}, on a graph. The term resistance distance was used because of the physical interpretation: place unit resistors on each edge of a graph $G$ and take the resistance distance, $r_{ij},$ between vertices $i$ and $j$ of $G$ to be the effective resistance between them. This novel parameter is in fact intrinsic to the graph and has some nice interpretations and applications in chemistry (see \cite{5,3} for details). It is well known that the resistance distance between two arbitrary vertices in an electronic network can be obtained in terms of the eigenvalues and eigenvectors of the Laplacian matrix associated with the network. One famous resistance distance-based parameter called the \textit{Kirchhoff index}, $Kf(G),$ was given by $Kf(G)=\sum_{i<j}r_{ij}$; see \cite{1}. Later it is shown \cite{3,9} that
$$
  Kf(G) =\sum_{i<j} r_{ij}=n\sum_{i=2}^n\frac{1}{\mu_i},
$$
where $0=\mu_1\leqslant\mu_2\leqslant \cdots \leqslant \mu_n \, (n\geqslant 2)$ are the eigenvalues of $L(G).$

In recent years, the \textit{normalized Laplacian}, $\mathcal{L}(G)$, which is consistent with the matrix in spectral geometry and random walks \cite{4}, has attracted more and more researchers' attention. One of the original motivations for defining the normalized Laplacian was to deal more
naturally with nonregular graphs. The normalized Laplacian is defined to be
$$
  \mathcal{L}=I-D^{\frac{1}{2}}(D^{-1}A)D^{-\frac{1}{2}}=D^{-\frac{1}{2}}LD^{-\frac{1}{2}}
$$ (with the convention that if the degree of vertex $v_i$ in $G$ is 0, then $(d_i)^{-\frac{1}{2}}=0;$ see \cite{4}). Thus it is easy to obtain that
\begin{eqnarray}\label{eq:1.1}
(\mathcal{L}(G))_{ij}=\left\{
                               \begin{array}{lll}
                               1, & \hbox{if $i=j$;} \\
                               -\frac{1}{\sqrt{d_id_j}}, & \hbox{if $i\neq j$ and $v_i$ is adjacent to $v_j;$} \\
                               0, & \hbox{otherwise,}
                               \end{array}
                            \right.
\end{eqnarray}
where $(\mathcal{L}(G))_{ij}$ denotes the $(i,j)$-entry of $\mathcal{L}(G)$. In 2005, Chen and Zhang \cite{10} showed that the resistance
distance can be expressed naturally in terms of the eigenvalues and eigenvectors of the normalized Laplacian and proposed another graph invariant, defined by $Kf^*(G)=\sum_{i<j}d_id_jr_{ij}$, which is called the \textit{degree-Kirchhoff index} (see \cite{12,800}). It is closely related to the corresponding spectrum of the normalized Laplacian (see Lemma 2.3 in the next section). The spectrum of $\mathcal{L}(G)$ is denoted by $S(G)=\{\lambda_1,\lambda_2,\ldots,\lambda_n\}$ with $0=\lambda_1\leqslant\lambda_2\leqslant\ldots\leqslant\lambda_n.$ It is well known that $G$ is connected if and only if $\lambda_2>0$.

It is well-known \cite{1} that $r_{ij}\leq d_{ij}$ with equality if and only if there is a unique path connecting vertices $v_i$ and $v_j$ in $G$. As an immediate consequence, for a tree $G$, $Kf(G)=W(G)$ and $Kf^*(G)=Gut(G)$. Thus in the research on the Kirchhoff index and the degree Kirchhoff index of graphs, it is primarily of interest in the case of cycle-containing graphs. Up to now, closed-form formula for Kirchhoff index have been given for some classes of graphs, such as cycles \cite{001}, circulant graphs \cite{003}, composite graphs \cite{004} and some other topics on Kirchhoff index of graphs may be referred to \cite{500,E-A-J,22s,hl,p1,p2,p3,q,11s,33s,ZZ3,005,006,060,Y-W-Z} and references therein.

A polyomino system is a finite 2-connected plane graph such that each interior face (or say a cell) is surrounded by a regular square of length one. Polyominos have a long and rich history, they now constitute one of the most popular subjects in mathematical recreations, and have found interest among mathematicians, physicists, biologists, and computer scientists as well. So, several molecular structure descriptors based on network structural descriptors, have been introduced, see for instance \cite{35}. In particular, in the last decade a great amount of work devoted to calculating the Kirchhoff index of linear polyominoes-like networks, have been published, one may be referred to \cite{34} and references therein for some details.
Let $B_n$ denote a linear polyomino chain with $n$ squares as depicted in Fig. 1. Then it is routine to check that $|V(B_n)|=2n+2$ and $|E(B_n)|=3n+1$.
\begin{figure}[h!]
\begin{center}
\psfrag{1}{$u_1$}\psfrag{2}{$u_2$}\psfrag{3}{$u_3$}\psfrag{4}{$u_n$}
\psfrag{5}{$u_{n+1}$}
\psfrag{a}{$v_1$}\psfrag{b}{$v_2$}\psfrag{c}{$v_3$}\psfrag{d}{$v_n$}
\psfrag{e}{$v_{n+1}$}
\psfrag{k}{$\cdots$}\psfrag{l}{$B_n$}
\includegraphics[width=80mm]{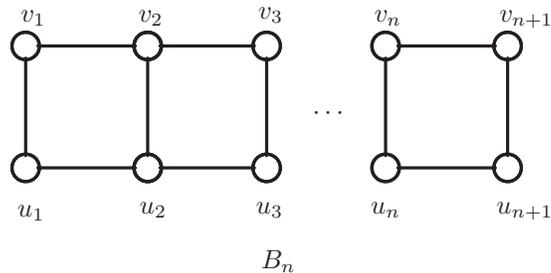} \\
  \caption{ Graph $B_n$ and its labeled vertices.} %\label{*}
\end{center}
\end{figure}

\begin{figure}[h!]
\begin{center}
\psfrag{1}{$u_1$}\psfrag{2}{$u_2$}\psfrag{3}{$u_3$}\psfrag{4}{$u_n$}
\psfrag{5}{$u_{n+1}$}\psfrag{0}{$u'_1$}\psfrag{6}{$u'_2$}\psfrag{7}{$u'_3$}
\psfrag{8}{$u'_n$}\psfrag{9}{$u'_{n+1}$}
\psfrag{a}{$v_1$}\psfrag{b}{$v_2$}\psfrag{c}{$v_3$}\psfrag{d}{$v_n$}
\psfrag{e}{$v_{n+1}$}\psfrag{f}{$v'_1$}\psfrag{g}{$v'_2$}\psfrag{h}{$v'_3$}
\psfrag{i}{$v'_n$}\psfrag{j}{$v'_{n+1}$}
\psfrag{k}{$\cdots$}\psfrag{l}{$B^2_n$}
\includegraphics[width=120mm]{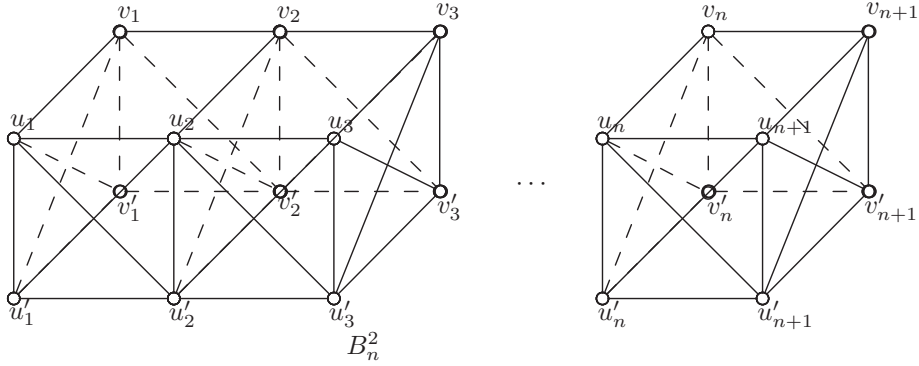} \\
  \caption{ Graph $B^2_n$ and its labeled vertices.} %\label{*}
\end{center}
\end{figure}

Given two graphs $G$ and $H$, the strong product of $G$ and $H$, denoted by $G\boxtimes H$, is the graph with vertex set $V(G)\times V(H)$, where two distinct vertices $(u_1, v_1)$ and $(u_2, v_2)$ are adjacent whenever $u_1$ and $u_2$ are equal or adjacent in $G$, and $v_1$ and $v_2$ are equal or adjacent in $H$. Specially, the strong product of $G$ and $K_2$ is called the strong prism of $G$. Very recently, Pan et al. \cite{p2,p3} determined some resistance distance-based invariants and number of spanning trees of graphs derived from the strong prism of some special graphs, such as the path $P_n$ and the cycle $C_n$. Li et al. \cite{p1} determined some resistance distance-based invariants and number of spanning trees of graphs derived from the strong prism of the star $S_n$. Along this line, we consider the strong prism of a linear polyomino chain $B_n$. Let $B_n^2$ be the strong prism of $B_n$ as depicted in Fig. 2. Obviously, $|V(B_n^2)|=4n+4$ and $|E(B_n^2)|=2(7n+3)$.

In this paper, motivated by \cite{S-B-B,10,800,J-H-L,hl,p1,p2,p3,700,X-Z-F,34}, explicit expressions for degree-Kirchhoff index and number of spanning trees of $B^2_n$ are determined, respectively. Furthermore, we are surprised to see that the degree-Kirchhoff index of $B^2_n$ is approximately one eighth of its Gutman index.

\section{\normalsize Preliminaries}\setcounter{equation}{0}
Throughout this paper, we shall denote by $\Phi(B) = \det(xI - B)$ the \textit{characteristic polynomial} of the square matrix $B$. In particular, if $B = \mathcal{L}(G)$, we write $\Phi(\mathcal{L}(G))$ by $\Psi(G; x)$ and call $\Psi(G; x)$ the \textit{normalized Laplacian characteristic polynomial} of $G$.

Let $V_1=\{u_1,u_2,\ldots,u_{n+1},v_1,v_2,\ldots, v_{n+1}\}, V_2=\{u'_1,u'_2,\ldots, u'_{n+1},v'_1,v'_2,\ldots, v'_{n+1}\}$. Then by a suitable arrangement of vertices in $B_n^2$, $\mathcal{L}(B_n^2)$ can be written as the following block matrix
$$
\mathcal{L}(G)=
\left(
\begin{tabular}{c|c}
  %\hline
  % after \\: \hline or \cline{col1-col2} \cline{col3-col4} ...
\text{$\begin{array}{c}
      \mathcal{L}_{V_{11}} \\
    \end{array}$}
&\text{$\begin{array}{c}
          \mathcal{L}_{V_{12}}\\
       \end{array}$}
\\\hline
\text{$\begin{array}{c}
       \mathcal{L}_{V_{21}}  \\
    \end{array}$}
&\text{$\begin{array}{c}
        \mathcal{L}_{V_{22}} \\
       \end{array}$}
\end{tabular}
\right),
$$
where $\mathcal{L}_{V_{ij}}$ is the submatrix formed by rows corresponding to vertices in $V_i$ and columns corresponding to vertices in $V_j$ for $i,j=1,2$. Owing to the symmetry construction of the graph $B_n^2$, it is obvious
that $\mathcal{L}_{V_{11}}(B_n^2)=\mathcal{L}_{V_{22}}(B_n^2)$ and $\mathcal{L}_{V_{12}}(B_n^2)=\mathcal{L}_{V_{21}}(B_n^2)$. Let
\[\label{eq:2.01}
T=\left(
\begin{tabular}{c|c}
  %\hline
  % after \\: \hline or \cline{col1-col2} \cline{col3-col4} ...
\text{$\begin{array}{c}
      \frac{1}{\sqrt{2}}I_{2n+2} \\
    \end{array}$}
&\text{$\begin{array}{c}
         \frac{1}{\sqrt{2}}I_{2n+2}\\
       \end{array}$}
\\\hline
\text{$\begin{array}{c}
       \frac{1}{\sqrt{2}}I_{2n+2}\\
    \end{array}$}
&\text{$\begin{array}{c}
        -\frac{1}{\sqrt{2}}I_{2n+2}\\
       \end{array}$}
\end{tabular}
\right)
\]
be the block matrix such that the blocks have the same dimension as the corresponding blocks in $\mathcal{L}(G)$. By a simple calculation, one can see that
\begin{eqnarray}\label{eq:2.1}
T\mathcal{L}(B_n^2)T=\left(
\begin{tabular}{c|c}
  %\hline
  % after \\: \hline or \cline{col1-col2} \cline{col3-col4} ...
\text{$\begin{array}{c}
      \mathcal{L}_A(B_n^2) \\
    \end{array}$}
&\text{$\begin{array}{c}
          \bf{0}\\
       \end{array}$}
\\\hline
\text{$\begin{array}{c}
        \bf{0}  \\
    \end{array}$}
&\text{$\begin{array}{c}
        \mathcal{L}_S(B_n^2)\\
       \end{array}$}
\end{tabular}
\right),
\end{eqnarray}
where $\mathcal{L}_A(B_n^2)=\mathcal{L}_{V_{11}}+\mathcal{L}_{V_{12}}$ and $\mathcal{L}_S(B_n^2)=\mathcal{L}_{V_{11}}-\mathcal{L}_{V_{12}}$.

Thus similar to the decomposition theorem obtained in \cite{hl,p1,34}, we can obtain the following decomposition theorem of the normalized Laplacian polynomial.
\begin{thm}
Let $\mathcal{L}(G), \mathcal{L}_A(G), \mathcal{L}_S(G)$ be defined as above. Then
$$
\Psi(G;x)=\Phi(\mathcal{L}_A(G))\cdot\Phi(\mathcal{L}_S(G)).
$$
\end{thm}

\begin{lem}[\cite{4}]
Let $G$ be an $n$-vertex connected graph with $m$ edges, then $\prod_{i=1}^nd_i\prod_{k=2}^n\lambda_k=2m\tau(G),$
where $\tau(G)$ is the number of spanning trees of $G.$
\end{lem}
\begin{lem}[\cite{10}]
Let $G$ be an $n$-vertex connected graph with $m$ edges, then
$
   Kf^*(G)=2m\sum_{i=2}^n\frac{1}{\lambda_i}.
$
\end{lem}
\section{\normalsize Degree-Kirchhoff index and the number of spanning trees of linear polyomino chain }\setcounter{equation}{0}
In this section, we first determine the normalized Laplacian eigenvalues of $B^2_n$ according to Theorem 2.1. Then we provide a complete description of the product of the normalized Laplacian eigenvalues and the sum of the normalized Laplacian eigenvalues' reciprocals which will be used in computing the degree-Kirchhoff index and the number of spanning trees of $B^2_n$. Finally, we show that the degree-Kirchhoff index of $B^2_n$ is approximately one half of its Gutman index.

For convenience, we abbreviate $\mathcal{L}_A(B^2_n)$ and $\mathcal{L}_S(B^2_n)$ to $\mathcal{L}_A$ and $\mathcal{L}_S$, respectively. We label the vertices of $B^2_n$ as depicted in Fig.~2. Obviously,
\begin{align*}
&\mathcal{L}_{11}(B^2_n)\\
=&{\left(
                       \begin{array}{cccccccccccccc}
                         1 & -\frac{1}{\sqrt{35}} & 0 & \cdots & 0 & 0 & 0 & -\frac{1}{5} & 0 & 0 & \cdots & 0 & 0 & 0 \\[3pt]
                         -\frac{1}{\sqrt{35}} & 1 & -\frac{1}{7} & \cdots & 0 & 0 & 0 & 0 & -\frac{1}{7} & 0 & \cdots & 0 & 0 & 0 \\[3pt]
                         0 & -\frac{1}{7} & 1 & \cdots & 0 & 0 & 0 & 0 & 0 & -\frac{1}{7} & \cdots & 0 & 0 & 0 \\[3pt]
                         \vdots & \vdots & \vdots & \ddots & \vdots & \vdots & \vdots & \vdots & \vdots & \vdots & \ddots & \vdots & \vdots & \vdots \\[3pt]
                         0 & 0 & 0 & \cdots & 1 & -\frac{1}{7} & 0 & 0 & 0 & 0 & \cdots & -\frac{1}{7} & 0 & 0 \\[3pt]
                         0 & 0 & 0 & \cdots & -\frac{1}{7} & 1 & -\frac{1}{\sqrt{35}} & 0 & 0 & 0 & \cdots & 0 & -\frac{1}{7} & 0 \\[3pt]
                         0 & 0 & 0 & \cdots & 0 & -\frac{1}{\sqrt{35}} & 1 & 0 & 0 & 0 & \cdots & 0 & 0 & -\frac{1}{5} \\[3pt]
                         -\frac{1}{5} & 0 & 0 & \cdots & 0 & 0 & 0 & 1 & -\frac{1}{\sqrt{35}} & 0 & \cdots & 0 & 0 & 0 \\[3pt]
                         0 & -\frac{1}{7} & 0 & \cdots & 0 & 0 & 0 & -\frac{1}{\sqrt{35}} & 1 & -\frac{1}{7} & \cdots & 0 & 0 & 0 \\[3pt]
                         0 & 0 & -\frac{1}{7} & \cdots & 0 & 0 & 0 & 0 & -\frac{1}{7} & 1 & \cdots & 0 & 0 & 0 \\[3pt]
                         \vdots & \vdots & \vdots & \ddots & \vdots & \vdots & \vdots & \vdots & \vdots & \vdots & \ddots & \vdots & \vdots & \vdots \\[3pt]
                         0 & 0 & 0 & \cdots & -\frac{1}{7} & 0 & 0 & 0 & 0 & 0 & \cdots & 1 & -\frac{1}{7} & 0 \\[3pt]
                         0 & 0 & 0 & \cdots & 0 & -\frac{1}{7} & 0 & 0 & 0 & 0 & \cdots & -\frac{1}{7} & 1 & -\frac{1}{\sqrt{35}} \\[3pt]
                         0 & 0 & 0 & \cdots & 0 & 0 & -\frac{1}{5} & 0 & 0 & 0 & \cdots & 0 & -\frac{1}{\sqrt{35}} & 1
                       \end{array}
                     \right)}_{(2n+2)\times (2n+2)}
\end{align*}
and
\begin{align*}
&\mathcal{L}_{12}(B^2_n)\\
=&{\left(
                       \begin{array}{cccccccccccccc}
                         -\frac{1}{5} & -\frac{1}{\sqrt{35}} & 0 & \cdots & 0 & 0 & 0 & -\frac{1}{5} & 0 & 0 & \cdots & 0 & 0 & 0 \\[3pt]
                         -\frac{1}{\sqrt{35}} & -\frac{1}{7} & -\frac{1}{7} & \cdots & 0 & 0 & 0 & 0 & -\frac{1}{7} & 0 & \cdots & 0 & 0 & 0 \\[3pt]
                         0 & -\frac{1}{7} & -\frac{1}{7} & \cdots & 0 & 0 & 0 & 0 & 0 & -\frac{1}{7} & \cdots & 0 & 0 & 0 \\[3pt]
                         \vdots & \vdots & \vdots & \ddots & \vdots & \vdots & \vdots & \vdots & \vdots & \vdots & \ddots & \vdots & \vdots & \vdots \\[3pt]
                         0 & 0 & 0 & \cdots & -\frac{1}{7} & -\frac{1}{7} & 0 & 0 & 0 & 0 & \cdots & -\frac{1}{7} & 0 & 0 \\[3pt]
                         0 & 0 & 0 & \cdots & -\frac{1}{7} & -\frac{1}{7} & -\frac{1}{\sqrt{35}} & 0 & 0 & 0 & \cdots & 0 & -\frac{1}{7} & 0 \\[3pt]
                         0 & 0 & 0 & \cdots & 0 & -\frac{1}{\sqrt{35}} & -\frac{1}{5} & 0 & 0 & 0 & \cdots & 0 & 0 & -\frac{1}{5} \\[3pt]
                         -\frac{1}{5} & 0 & 0 & \cdots & 0 & 0 & 0 & -\frac{1}{5} & -\frac{1}{\sqrt{35}} & 0 & \cdots & 0 & 0 & 0 \\[3pt]
                         0 & -\frac{1}{7} & 0 & \cdots & 0 & 0 & 0 & -\frac{1}{\sqrt{35}} & -\frac{1}{7} & -\frac{1}{7} & \cdots & 0 & 0 & 0 \\[3pt]
                         0 & 0 & -\frac{1}{7} & \cdots & 0 & 0 & 0 & 0  & -\frac{1}{7} & -\frac{1}{7} & \cdots & 0 & 0 & 0 \\[3pt]
                         \vdots & \vdots & \vdots & \ddots & \vdots & \vdots & \vdots & \vdots & \vdots & \vdots & \ddots & \vdots & \vdots & \vdots \\[3pt]
                         0 & 0 & 0 & \cdots & -\frac{1}{7} & 0 & 0 & 0 & 0 & 0 & \cdots & -\frac{1}{7} & -\frac{1}{7} & 0 \\[3pt]
                         0 & 0 & 0 & \cdots & 0 & -\frac{1}{7} & 0 & 0 & 0 & 0 & \cdots & -\frac{1}{7} & -\frac{1}{7} & -\frac{1}{\sqrt{35}} \\[3pt]
                         0 & 0 & 0 & \cdots & 0 & 0 & -\frac{1}{5} & 0 & 0 & 0 & \cdots & 0 & -\frac{1}{\sqrt{35}} & -\frac{1}{5}
                       \end{array}
                     \right)}_{(2n+2)\times (2n+2)}.
\end{align*}
Since $\mathcal{L}_{A}=\mathcal{L}_{11}(B^2_n)+\mathcal{L}_{12}(B^2_n)$ and $\mathcal{L}_{S}=\mathcal{L}_{11}(B^2_n)-\mathcal{L}_{12}(B^2_n)$, thus
\begin{align*}
&\mathcal{L}_{A}\\
=&2{\left(
                       \begin{array}{cccccccccccccc}
                         \frac{2}{5} & -\frac{1}{\sqrt{35}} & 0 & \cdots & 0 & 0 & 0 & -\frac{1}{5} & 0 & 0 & \cdots & 0 & 0 & 0 \\[3pt]
                         -\frac{1}{\sqrt{35}} & \frac{3}{7} & -\frac{1}{7} & \cdots & 0 & 0 & 0 & 0 & -\frac{1}{7} & 0 & \cdots & 0 & 0 & 0 \\[3pt]
                         0 & -\frac{1}{7} & \frac{3}{7} & \cdots & 0 & 0 & 0 & 0 & 0 & -\frac{1}{7} & \cdots & 0 & 0 & 0 \\[3pt]
                         \vdots & \vdots & \vdots & \ddots & \vdots & \vdots & \vdots & \vdots & \vdots & \vdots & \ddots & \vdots & \vdots & \vdots \\[3pt]
                         0 & 0 & 0 & \cdots &  \frac{3}{7} & -\frac{1}{7} & 0 & 0 & 0 & 0 & \cdots & -\frac{1}{7} & 0 & 0 \\[3pt]
                         0 & 0 & 0 & \cdots & -\frac{1}{7} & \frac{3}{7} & -\frac{1}{\sqrt{35}} & 0 & 0 & 0 & \cdots & 0 & -\frac{1}{7} & 0 \\[3pt]
                         0 & 0 & 0 & \cdots & 0 & -\frac{1}{\sqrt{35}} &  \frac{2}{5} & 0 & 0 & 0 & \cdots & 0 & 0 & -\frac{1}{5} \\[3pt]
                         -\frac{1}{5} & 0 & 0 & \cdots & 0 & 0 & 0 & \frac{2}{5} & -\frac{1}{\sqrt{35}} & 0 & \cdots & 0 & 0 & 0 \\[3pt]
                         0 & -\frac{1}{7} & 0 & \cdots & 0 & 0 & 0 & -\frac{1}{\sqrt{35}} & \frac{3}{7} & -\frac{1}{7} & \cdots & 0 & 0 & 0 \\[3pt]
                         0 & 0 & -\frac{1}{7} & \cdots & 0 & 0 & 0 & 0 & -\frac{1}{7} & \frac{3}{7} & \cdots & 0 & 0 & 0 \\[3pt]
                         \vdots & \vdots & \vdots & \ddots & \vdots & \vdots & \vdots & \vdots & \vdots & \vdots & \ddots & \vdots & \vdots & \vdots \\[3pt]
                         0 & 0 & 0 & \cdots & -\frac{1}{7} & 0 & 0 & 0 & 0 & 0 & \cdots & \frac{3}{7} & -\frac{1}{7} & 0 \\[3pt]
                         0 & 0 & 0 & \cdots & 0 & -\frac{1}{7} & 0 & 0 & 0 & 0 & \cdots & -\frac{1}{7} & \frac{3}{7} & -\frac{1}{\sqrt{35}} \\[3pt]
                         0 & 0 & 0 & \cdots & 0 & 0 & -\frac{1}{5} & 0 & 0 & 0 & \cdots & 0 & -\frac{1}{\sqrt{35}} & \frac{2}{5}
                       \end{array}
                     \right)}_{(2n+2)\times (2n+2)}
\end{align*}
and
$$
\mathcal{L}_{S}={\left(
                       \begin{array}{cccccccccccccc}
                         \frac{6}{5} & 0 & 0 & \cdots & 0 & 0 & 0 & 0 & 0 & 0 & \cdots & 0 & 0 & 0 \\[3pt]
                         0 & \frac{8}{7} & 0 & \cdots & 0 & 0 & 0 & 0 & 0 & 0 & \cdots & 0 & 0 & 0 \\[3pt]
                         0 & 0 & \frac{8}{7} & \cdots & 0 & 0 & 0 & 0 & 0 & 0 & \cdots & 0 & 0 & 0 \\[3pt]
                         \vdots & \vdots & \vdots & \ddots & \vdots & \vdots & \vdots & \vdots & \vdots & \ddots & \vdots & \vdots & \vdots \\[3pt]
                         0 & 0 & 0 & \cdots & \frac{8}{7} & 0 & 0 & 0 & 0 & 0 & \cdots & 0 & 0 & 0 \\[3pt]
                         0 & 0 & 0 & \cdots & 0 & \frac{8}{7} & 0 & 0 & 0 & 0 & \cdots & 0 & 0 & 0 \\[3pt]
                         0 & 0 & 0 & \cdots & 0 & 0 & \frac{6}{5} & 0 & 0 & 0 & \cdots & 0 & 0 & 0 \\[3pt]
                         0 & 0 & 0 & \cdots & 0 & 0 & 0 & \frac{6}{5} & 0 & 0 & \cdots & 0 & 0 & 0 \\[3pt]
                         0 & 0 & 0 & \cdots & 0 & 0 & 0 & 0 & \frac{8}{7} & 0 & \cdots & 0 & 0 & 0 \\[3pt]
                         0 & 0 & 0 & \cdots & 0 & 0 & 0 & 0 & 0 & \frac{8}{7} & \cdots & 0 & 0 & 0 \\[3pt]
                         \vdots & \vdots & \vdots & \ddots & \vdots & \vdots & \vdots & \vdots & \vdots & \vdots & \ddots & \vdots & \vdots & \vdots \\[3pt]
                         0 & 0 & 0 & \cdots & 0 & 0 & 0 & 0 & 0 & 0 & \cdots & \frac{8}{7} & 0 & 0\\[3pt]
                         0 & 0 & 0 & \cdots & 0 & 0 & 0 & 0 & 0 & 0 & \cdots & 0 & \frac{8}{7} & 0 \\[3pt]
                         0 & 0 & 0 & \cdots & 0 & 0 & 0 & 0 & 0 & 0 & \cdots & 0 & 0 & \frac{6}{5}
                       \end{array}
                     \right)}_{(2n+2)\times (2n+2)}.
$$
Therefore, by Theorem 2.1, we have the normalized Laplacian eigenvalues of $B^2_n$ consists of the eigenvalues of $\mathcal{L}_A$ and $\mathcal{L}_S$. Since $\mathcal{L}_{S}$ is a diagonal matrix, one can easily see that $\frac{6}{5}$ with multiplicity 4 and $\frac{8}{7}$ with multiplicity $2n-2$ are all the eigenvalues of $\mathcal{L}_S$. Next, we provide a complete description of the product of the eigenvalues of $\mathcal{L}_A$ and the sum of the reciprocals of the eigenvalues of $\mathcal{L}_A$ which will be used in computing the degree-Kirchhoff index and the number of spanning trees of $B^2_n$.

Let
\begin{align*}
C={\left(
         \begin{array}{ccccccc}
         \frac{2}{5} & -\frac{1}{\sqrt{35}} & 0 & \cdots & 0 & 0 & 0 \\[3pt]
         -\frac{1}{\sqrt{35}} & \frac{3}{7} & -\frac{1}{7} & \cdots & 0 & 0 & 0\\[3pt]
         0 & -\frac{1}{7} & \frac{3}{7} & \cdots & 0 & 0 & 0 \\[3pt]
         \vdots & \vdots & \vdots & \ddots & \vdots & \vdots & \vdots \\[3pt]
         0 & 0 & 0 & \cdots &  \frac{3}{7} & -\frac{1}{7} & 0 \\[3pt]
         0 & 0 & 0 & \cdots & -\frac{1}{7} & \frac{3}{7} & -\frac{1}{\sqrt{35}} \\[3pt]
         0 & 0 & 0 & \cdots & 0 & -\frac{1}{\sqrt{35}} &  \frac{2}{5}
         \end{array}
         \right)}_{(n+1)\times (n+1)}
\end{align*}
and
\begin{align*}
D={\left(
        \begin{array}{ccccccc}
        -\frac{1}{5} & 0 & 0 & \cdots & 0 & 0 & 0 \\[3pt]
        0 & -\frac{1}{7} & 0 & \cdots & 0 & 0 & 0 \\[3pt]
        0 & 0 & -\frac{1}{7} & \cdots & 0 & 0 & 0 \\[3pt]
        \vdots & \vdots & \vdots & \ddots & \vdots & \vdots & \vdots \\[3pt]
        0 & 0 & 0 & \cdots & -\frac{1}{7} & 0 & 0 \\[3pt]
        0 & 0 & 0 & \cdots & 0 & -\frac{1}{7} & 0 \\[3pt]
        0 & 0 & 0 & \cdots & 0 & 0 & -\frac{1}{5}
        \end{array}
        \right)}_{(n+1)\times (n+1)}.
\end{align*}
Then
$\frac{1}{2}\mathcal{L}_A$ can be written as the following block matrix
$$
\frac{1}{2}\mathcal{L}_A=
\left(
\begin{tabular}{c|c}
\text{$\begin{array}{c}
      C \\
    \end{array}$}
&\text{$\begin{array}{c}
       D\\
       \end{array}$}
\\\hline
\text{$\begin{array}{c}
      D\\
    \end{array}$}
&\text{$\begin{array}{c}
      C \\
       \end{array}$}
\end{tabular}
\right).
$$
Let
$$
T=\left(
\begin{tabular}{c|c}
  %\hline
  % after \\: \hline or \cline{col1-col2} \cline{col3-col4} ...
\text{$\begin{array}{c}
      \frac{1}{\sqrt{2}}I_{n+1} \\
    \end{array}$}
&\text{$\begin{array}{c}
         \frac{1}{\sqrt{2}}I_{n+1}\\
       \end{array}$}
\\\hline
\text{$\begin{array}{c}
       \frac{1}{\sqrt{2}}I_{n+1}  \\
    \end{array}$}
&\text{$\begin{array}{c}
        -\frac{1}{\sqrt{2}}I_{n+1}\\
       \end{array}$}
\end{tabular}
\right)
$$
be the block matrix such that the blocks have the same dimension as the corresponding blocks in $\frac{1}{2}\mathcal{L}_A$. Then
\begin{eqnarray*}
T\big(\frac{1}{2}\mathcal{L}_A\big)T=\left(
\begin{tabular}{c|c}
  %\hline
  % after \\: \hline or \cline{col1-col2} \cline{col3-col4} ...
\text{$\begin{array}{c}
      C+D \\
    \end{array}$}
&\text{$\begin{array}{c}
          \bf{0}\\
       \end{array}$}
\\\hline
\text{$\begin{array}{c}
        \bf{0}  \\
    \end{array}$}
&\text{$\begin{array}{c}
        C-D\\
       \end{array}$}
\end{tabular}
\right).
\end{eqnarray*}
Let $M=C+D$ and $N=C-D$. It is easy to check that the eigenvalues of $\frac{1}{2}\mathcal{L}_A$ consists of the eigenvalues of $M$ and $N$. Suppose that the eigenvalues of $M$ and $N$ are respectively, denoted by $\alpha_i$ and $\beta_j\ (i,j=1,2,\ldots,n+1)$ with $\alpha_1\leqslant\alpha_2\leqslant\ldots\leqslant\alpha_{n+1}, \beta_1\leqslant\beta_2\leqslant\ldots\leqslant\beta_{n+1}$. Then the eigenvalues of $\mathcal{L}_A$ are $2\alpha_1,2\alpha_2,\ldots,2\alpha_{n+1},2\beta_1,2\beta_2,\ldots,2\beta_{n+1}$. Since all the eigenvalues of $\mathcal{L}_A$ are the normalized Laplacian eigenvalues of $B^2_n$, we get $\alpha_1, \beta_1\geqslant0$.

Let $\xi=(\sqrt{5},\sqrt{7},\dots,\sqrt{7},\sqrt{5})^T$, then we have $M\xi=0$, thus $\alpha_1=0$, and $\alpha_i>0,  \beta_j>0$ for $2\leqslant i\leqslant n+1, 1\leqslant j\leqslant n+1.$ Note that $|E_{B^2_n}|=2(7n+3)$, the following lemma is an immediate consequence of Lemma 2.3.
\begin{lem}
Let $B^2_n$ be the strong prism of a linear polyomino chain with $n$ squares. Let
\begin{align*}
M={\left(
         \begin{array}{ccccccc}
         \frac{1}{5} & -\frac{1}{\sqrt{35}} & 0 & \cdots & 0 & 0 & 0 \\[3pt]
         -\frac{1}{\sqrt{35}} & \frac{2}{7} & -\frac{1}{7} & \cdots & 0 & 0 & 0\\[3pt]
         0 & -\frac{1}{7} & \frac{2}{7} & \cdots & 0 & 0 & 0 \\[3pt]
         \vdots & \vdots & \vdots & \ddots & \vdots & \vdots & \vdots \\[3pt]
         0 & 0 & 0 & \cdots &  \frac{2}{7} & -\frac{1}{7} & 0 \\[3pt]
         0 & 0 & 0 & \cdots & -\frac{1}{7} & \frac{2}{7} & -\frac{1}{\sqrt{35}} \\[3pt]
         0 & 0 & 0 & \cdots & 0 & -\frac{1}{\sqrt{35}} &  \frac{1}{5}
         \end{array}
         \right)}_{(n+1)\times (n+1)}
\end{align*}
and
\begin{align*}
N={\left(
         \begin{array}{ccccccc}
         \frac{3}{5} & -\frac{1}{\sqrt{35}} & 0 & \cdots & 0 & 0 & 0 \\[3pt]
         -\frac{1}{\sqrt{35}} & \frac{4}{7} & -\frac{1}{7} & \cdots & 0 & 0 & 0\\[3pt]
         0 & -\frac{1}{7} & \frac{4}{7} & \cdots & 0 & 0 & 0 \\[3pt]
         \vdots & \vdots & \vdots & \ddots & \vdots & \vdots & \vdots \\[3pt]
         0 & 0 & 0 & \cdots &  \frac{4}{7} & -\frac{1}{7} & 0 \\[3pt]
         0 & 0 & 0 & \cdots & -\frac{1}{7} & \frac{4}{7} & -\frac{1}{\sqrt{35}} \\[3pt]
         0 & 0 & 0 & \cdots & 0 & -\frac{1}{\sqrt{35}} &  \frac{3}{5}
         \end{array}
         \right)}_{(n+1)\times (n+1)}.
\end{align*}
Then
\begin{eqnarray*}
Kf^*(B^2_n)=4(7n+3)\left(4\times\frac{5}{6}+(2n-2)\frac{7}{8}+\frac{1}{2}\sum_{i=2}^{n+1}\frac{1}{\alpha_i}+\frac{1}{2}\sum_{j=1}^{n+1}\frac{1}{\beta_j}\right),
\end{eqnarray*}
where $0=\alpha_1<\alpha_2\leqslant\ldots\leqslant\alpha_{n+1}, 0<\beta_1\leqslant\beta_2\leqslant\ldots\leqslant\beta_{n+1}$ are the eigenvalues of $M$ and $N$, respectively.
\end{lem}
Based on the relationship between the roots and coefficients of $\Phi(M)$ (resp. $\Phi(N)$), the formula of $\sum_{i=2}^{n+1}\frac{1}{\alpha_i}$ (resp. $\sum_{j=1}^{n+1}\frac{1}{\beta_j}$) is derived in the next two lemmas.
\begin{lem} Let $\alpha_i\, (i=1,2,\ldots,n+1)$ be defined as above. Then
\begin{eqnarray*}
\sum_{i=2}^{n+1}\frac{1}{\alpha_i}=\frac{49n^3+63n^2+38n}{6(7n+3)}.
\end{eqnarray*}
\end{lem}
\begin{proof}
Suppose that $\Phi(M)=x^{n+1}+a_1x^n+\cdots+a_{n-1}x^2+a_nx=x(x^n+a_1x^{n-1}+\cdots+a_{n-1}x+a_n).$ Then $\alpha_i\, (i=2,3,\ldots,n+1)$ satisfies the following equation
$$
x^n+a_1x^{n-1}+\cdots+a_{n-1}x+a_n=0.
$$
and so $\frac{1}{\alpha_i}\, (i=2,3,\ldots,n+1)$ satisfies the following equation
$$
a_nx^n+a_{n-1}x^{n-1}+\cdots+a_1x+1=0.
$$
By Vieta's Theorem, we have
\begin{eqnarray}\label{eq:3.1}
\sum_{i=2}^{n+1}\frac{1}{\alpha_i}=\frac{(-1)^{n-1}a_{n-1}}{(-1)^na_n}.
\end{eqnarray}
For $1\leqslant i\leqslant n$, let $M_i$ be the $i$th order principal submatrix formed by the first $i$ rows and columns of $M$ and $m_i:=\det M_i$. We first derive the formula of $m_i$, which will be used in calculating $(-1)^{n-1}a_{n-1}$ and $(-1)^na_n$.
\begin{claim}
For $1\leqslant i\leqslant n, m_i=\frac{1}{5}\cdot\left(\frac{1}{7}\right)^{i-1}$.
\end{claim}
\noindent{\bf Proof of Claim 1.}\
It is easy to know that $m_1=\frac{1}{5}, m_2=\frac{1}{35}.$ For $3\leqslant i\leqslant n$, expanding $\det M_i$ with regard to its last row, we have
$$
m_i=\frac{2}{7}m_{i-1}-\frac{1}{49}m_{i-2}.
$$
Then the characteristic equation of $\{m_i\}_{i\geqslant1}$ is $x^2=\frac{2}{7}x-\frac{1}{49}$, whose roots are $x_1=x_2=\frac{1}{7}$. Suppose that
\begin{eqnarray}\label{eq:3.2}
m_i=(y_1i+y_2)\left(\frac{1}{7}\right)^i,
\end{eqnarray}
then the initial conditions, $m_1=\frac{1}{5}$ and $m_2=\frac{1}{35},$ lead to the system of equations
\begin{eqnarray*}
\left\{
  \begin{array}{l}
    \displaystyle \frac{1}{7}(y_1+y_2)=\frac{1}{5}, \\[6pt]
    \displaystyle \frac{1}{49}(2y_1+y_2)=\frac{1}{35}.
  \end{array}
\right.
\end{eqnarray*}
Solving it, we get $y_1=0, y_2=\frac{7}{5}.$ Thus Claim 1 follows by substituting $y_1$ and $y_2$ back into (\ref{eq:3.2}), as desired.\\

Then based on Claim 1, we give the expressions of  $(-1)^na_n$ and $(-1)^{n-1}a_{n-1}$ by the following two claims, respectively. For the sake of convenience, we let the diagonal entries of $M$ be $k_{ii}$ and $m_0$ be 1.
\begin{claim}
$(-1)^na_n=\frac{7n+3}{25}\cdot\left(\frac{1}{7}\right)^{n-1}.$
\end{claim}
\noindent{\bf Proof of Claim 2.}\
Since the number $(-1)^{n}a_{n}$ is the sum of those principal minors of $M$ which have $n$ rows and columns, we have
\begin{eqnarray}
  (-1)^{n}a_{n}
            &=&\sum_{i=1}^{n+1}\left|\begin{array}{cccccccc}
                                k_{11} & -\frac{1}{\sqrt{35}} & \cdots & 0 & 0& 0 & 0 & 0 \\[3pt]
                                -\frac{1}{\sqrt{35}} & k_{22} & \cdots & 0 & 0 & 0 & 0& 0\\[3pt]
                                \vdots & \vdots & \ddots & \vdots & \vdots & \vdots & \vdots & \vdots \\[3pt]
                                0 & 0 & \cdots & k_{i-1,i-1} & 0 & 0 & \cdots & 0 \\[3pt]
                                0 & 0 & \cdots & 0 & k_{i+1,i+1} & -\frac{1}{7} & \cdots & 0 \\[3pt]
                               \vdots & \vdots & \cdots & \vdots & \vdots & \ddots & \vdots & \vdots \\[3pt]
                               0 & 0 & \cdots & 0 & 0 & \cdots & k_{n,n}&-\frac{1}{\sqrt{35}}\\[3pt]
                                0 & 0 & \cdots & 0 & 0 & \cdots & -\frac{1}{\sqrt{35}} &k_{n+1,n+1}
                              \end{array}\right|\notag\\[3pt]
      &=&\sum_{i=1}^{n+1} \left(\left|\begin{array}{cccc}
                                    k_{11} & -\frac{1}{\sqrt{35}} & \cdots & 0 \\[3pt]
                                    -\frac{1}{\sqrt{35}} & k_{22} & \cdots & 0 \\[3pt]
                                    \vdots & \vdots & \ddots & \vdots\\[3pt]
                                    0 & 0 & \cdots & k_{i-1,i-1}
                                  \end{array}\right|\left|\begin{array}{cccc}
                                                                 k_{i+1,i+1} & -\frac{1}{7} & \cdots & 0 \\[3pt]
                                                                \vdots & \ddots & \vdots & \vdots  \\[3pt]
                                                                0 & \cdots & k_{n,n}&-\frac{1}{\sqrt{35}} \\[3pt]
                                                                 0 & \cdots & -\frac{1}{\sqrt{35}} &k_{n+1,n+1}
                                                               \end{array}
                                                               \right|\right).\label{eq:3.03}
\end{eqnarray}
Note that a permutation similarity transformation of a square matrix preserves its determinant. Together with the property of $M$, the right hand side of (\ref{eq:3.03}) is equal to $\det M_{n+1-i}$. By Claim 1, we have
\begin{eqnarray*}
(-1)^{n}a_{n}&=&\sum_{i=1}^{n+1}m_{i-1}m_{n+1-i}=2m_{n}+\sum_{i=2}^{n}m_{i-1}m_{n+1-i}\\
&=&2\cdot\frac{1}{5}\left(\frac{1}{7}\right)^{n-1}+\sum_{i=2}^{n}\frac{1}{5}\left(\frac{1}{7}\right)^{i-2}\cdot\frac{1}{5}\left(\frac{1}{7}\right)^{n-i}\\
&=&\frac{7n+3}{25}\cdot\left(\frac{1}{7}\right)^{n-1}.
\end{eqnarray*}
This completes the proof of Claim 2.
\begin{claim}
$
(-1)^{n-1}a_{n-1}=\frac{49n^3+63n^2+38n}{150}\cdot\left(\frac{1}{7}\right)^{n-1}.
$
\end{claim}
\noindent{\bf Proof of Claim 3.}\
Note that the number $(-1)^{n-1}a_{n-1}$ is the sum of those principal minors of $M$ which have $(n-1)$ rows and columns, hence  $(-1)^{n-1}a_{n-1}$ equals
\begin{align*}
&\sum_{1\leqslant i<j\leqslant n+1}\left|\begin{array}{cccccccccccc}
                                                            k_{11} & \frac{-1}{\sqrt{35}} & \cdots & 0 & 0 & 0 & \cdots & 0 & 0 & \cdots & 0 & 0 \\[3pt]
                                                            \frac{-1}{\sqrt{35}} & k_{22} & \cdots & 0 & 0 & 0 & \cdots& 0 & 0 & \cdots& 0 & 0 \\[3pt]
                                                            \vdots &\vdots&\ddots&\vdots&\vdots &\vdots &\cdots& \vdots & \vdots &\cdots &\vdots&\vdots \\[3pt]
                                                            0 & 0&\cdots&k_{i-1,i-1}&0 &0&\cdots &0& 0 &\cdots&0 & 0 \\[3pt]
                                                            0&0&\cdots&0&k_{i+1,i+1}&-\frac{1}{7}&\cdots&0& 0&\cdots& 0 & 0 \\[3pt]
                                                            0& 0& \cdots& 0 &-\frac{1}{7}& k_{i+2,i+2}&\cdots& 0 &0& \cdots&0 & 0\\[3pt]
                                                            \vdots&\vdots&\cdots&\vdots&\vdots&\vdots& \ddots&\vdots&\vdots&\cdots&\vdots &\vdots\\[3pt]
                                                            0& 0 & \cdots& 0& 0& 0 &\cdots& k_{j-1,j-1} & 0 &\cdots & 0 & 0 \\[3pt]
                                                            0 & 0 & \cdots& 0 & 0& 0& \cdots& 0&k_{j+1,j+1}&\cdots& 0& 0\\[3pt]
                                                            \vdots&\vdots&\cdots&\vdots&\vdots& \vdots&\cdots &\vdots&\vdots&\ddots&\vdots&\vdots\\[3pt]
                                                            0 & 0 & \cdots& 0& 0& 0& \cdots&0 & 0& \cdots& k_{n,n}& \frac{-1}{\sqrt{35}}\\[3pt]
                                                            0 &0& \cdots & 0& 0 & 0& \cdots& 0 & 0&\cdots&\frac{-1}{\sqrt{35}}& k_{n+1,n+1}
                                                          \end{array}
    \right|\\
\end{align*}
\begin{align}
=&\sum_{1\leqslant i<j\leqslant n+1}\left(\left|\begin{array}{cccc}
                                                     k_{11} & -\frac{1}{\sqrt{35}} & \cdots & 0 \\[3pt]
                                                     -\frac{1}{\sqrt{35}} & k_{22} & \cdots & 0 \\[3pt]
                                                     \vdots &\vdots&\ddots&\vdots\\[3pt]
                                                     0 & 0 &\cdots & k_{i-1,i-1}
                                                   \end{array}\right|
\left|\begin{array}{rrrr}
             k_{i+1,i+1}&-\frac{1}{7}&\cdots&0 \\[3pt]
            -\frac{1}{7}& k_{i+2,i+2}&\cdots& 0  \\[3pt]
             \vdots&\vdots& \ddots&\vdots\\[3pt]
             0& 0 &\cdots& k_{j-1,j-1}
           \end{array}\right|\right.\notag\\[3pt]
&\times\left.\left|\begin{array}{rrrr}
             k_{j+1,j+1}&\cdots& 0& 0\\[3pt]
             \vdots&\ddots&\vdots&\vdots \\[3pt]
              0& \cdots& k_{n,n}& -\frac{1}{\sqrt{35}}\\[3pt]
              0&\cdots&-\frac{1}{\sqrt{35}}& k_{n+1,n+1}
           \end{array}\right|\right).\label{eq:3.04}
\end{align}
Note that a permutation similarity transformation of a square matrix preserves its determinant. Hence, together with the property of $M$, the right hand side of (\ref{eq:3.04}) is equal to $\det M_{n+1-j}$. Thus,
\begin{eqnarray}\label{eq:3.3}
(-1)^{n-1}a_{n-1}=\sum_{1\leqslant i<j\leqslant n+1}m_{i-1}m_{n+1-j}\cdot\det P,
\end{eqnarray}
where
$$
P=\left(
    \begin{array}{rrrrrr}
      \frac{2}{7} & -\frac{1}{7} & 0 & \cdots & 0 & 0 \\[3pt]
      -\frac{1}{7} & \frac{2}{7} & -\frac{1}{7} & \cdots & 0 & 0 \\[3pt]
      0 & -\frac{1}{7} & \frac{2}{7} & \cdots & 0 & 0 \\[3pt]
      \vdots & \vdots & \vdots & \ddots & \vdots & \vdots \\[3pt]
      0 & 0 & 0 & \cdots & \frac{2}{7} & -\frac{1}{7} \\[3pt]
      0 & 0 & 0 & \cdots & -\frac{1}{7} &  \frac{2}{7}
    \end{array}
  \right)_{(j-i-1)\times(j-i-1)}.
$$
Straightforward computation yields that $\det P=\left(\frac{1}{7}\right)^{j-i-1}(j-i)$. Thus in view of (\ref{eq:3.3}), we have
\begin{eqnarray*}
(-1)^{n-1}a_{n-1}&=&\sum_{1\leqslant i<j\leqslant n+1}\left(\frac{1}{7}\right)^{j-i-1}(j-i)\cdot m_{i-1}m_{n+1-j}
\end{eqnarray*}
\begin{eqnarray*}
&=&\sum_{i=2}^{n}\left(\frac{1}{7}\right)^{n-i}(n+1-i)\cdot m_{i-1}+\sum_{j=2}^{n}\left(\frac{1}{7}\right)^{j-2}(j-1)\cdot m_{n+1-j}\\
&&+\sum_{2\leqslant i<j\leqslant n}\left(\frac{1}{7}\right)^{j-i-1}(j-i)\cdot m_{i-1}m_{n+1-j}+n\left(\frac{1}{7}\right)^{n-1}\\
&=&2\sum_{i=2}^{n}\left(\frac{1}{7}\right)^{n-i}(n+1-i)\cdot m_{i-1}+\sum_{2\leqslant i<j\leqslant n}\left(\frac{1}{7}\right)^{j-i-1}(j-i)\cdot m_{i-1}m_{n+1-j}+n\left(\frac{1}{7}\right)^{n-1}.
\end{eqnarray*}
By Claim 1, we have
$$
\sum_{i=2}^{n}\left(\frac{1}{7}\right)^{n-i}(n+1-i)\cdot m_{i-1}=\sum_{i=2}^{n}\left(\frac{1}{7}\right)^{n-i}\left(\frac{1}{7}\right)^{i-2}\frac{n+1-i}{5}
=\frac{n(n-1)}{10}\left(\frac{1}{7}\right)^{n-2}
$$
and
\begin{eqnarray*}
\sum_{2\leqslant i<j\leqslant n}\left(\frac{1}{7}\right)^{j-i-1}(j-i)\cdot m_{i-1}m_{n+1-j}
&=&\sum_{2\leqslant i<j\leqslant n}\left(\frac{1}{7}\right)^{j-i-1}\left(\frac{1}{7}\right)^{i-2}\left(\frac{1}{7}\right)^{n-j}\frac{j-i}{25}\\
&=&\frac{1}{25}\left(\frac{1}{7}\right)^{n-3}\sum_{2\leqslant i<j\leqslant n}(j-i)\\
&=&\frac{n(n-1)(n-2)}{150}\left(\frac{1}{7}\right)^{n-3}.
\end{eqnarray*}
Therefore,
\begin{eqnarray*}
(-1)^{n-1}a_{n-1}&=&\frac{n(n-1)}{5}\left(\frac{1}{7}\right)^{n-2}+\frac{n(n-1)(n-2)}{150}\left(\frac{1}{7}\right)^{n-3}+n\left(\frac{1}{7}\right)^{n-1}
\\&=&\frac{49n^3+63n^2+38n}{150}\cdot\left(\frac{1}{7}\right)^{n-1}.
\end{eqnarray*}
This completes the proof of Claim 3.

Substituting Claims 2-3 into (\ref{eq:3.1}) yields
$
\sum_{i=2}^{n+1}\frac{1}{\alpha_i}=\frac{49n^3+63n^2+38n}{6(7n+3)}$,
as desired.
\end{proof}

%Next we focus our attention on computing $\sum_{j=1}^{n+1}\frac{1}{\beta_j}$. According to the relationship between roots and coefficients of $\Phi(\mathcal{L}_S)$, we can obtain the following Lemma 3.3.
\begin{lem} Let $\beta_j\, (1\leqslant j\leqslant n+1)$ denote the eigenvalue of $N$ as above. Then
$$
\sum_{j=1}^{n+1}\frac{1}{\beta_j}=\frac{\sqrt{3}(7n-9)\left[\left(2+\sqrt{3}\right)^{n+1}+\left(2-\sqrt{3}\right)^{n+1}\right]-4\left[\left(2+\sqrt{3}\right)^{n}-\left(2-\sqrt{3}\right)^{n}\right]}{6\left[\left(2+\sqrt{3}\right)^{n+1}-\left(2-\sqrt{3}\right)^{n+1}\right]}+\frac{9}{2}.
$$
\end{lem}
\begin{proof}
Suppose that $\Phi(N)=x^{n+1}+b_1x^n+\cdots+b_nx+b_{n+1}.$ Then $\beta_j\, (j=1,2,\ldots,n+1)$ satisfies the following equation
$$
x^{n+1}+b_1x^n+\cdots+b_nx+b_{n+1}=0
$$
and so $\frac{1}{\beta_j}\, (j=1,2,\ldots,n+1)$ satisfies the following equation
$$
b_{n+1}x^{n+1}+b_{n}x^{n}+\cdots+b_1x+1=0.
$$
By Vieta's Theorem, we have
\begin{eqnarray}\label{eq:3.4}
\sum_{j=1}^{n+1}\frac{1}{\beta_j}=\frac{(-1)^nb_n}{(-1)^{n+1}b_{n+1}}=\frac{(-1)^nb_n}{\det N}.
\end{eqnarray}
For $1\leqslant i\leqslant n$, let $W_i$ be the $i$th order principal submatrix formed by the first $i$ rows and columns of $N$ and $w_i=\det W_i$. The following claim gives the formula of $w_i$, which will be used to calculate $(-1)^nb_n$ and $\det N$.
\begin{claim}
For $1\leqslant i\leqslant n, w_i=\frac{21+7\sqrt{3}}{30}\left(\frac{2+\sqrt{3}}{7}\right)^{i}+\frac{21-7\sqrt{3}}{30}\left(\frac{2-\sqrt{3}}{7}\right)^{i}$.
\end{claim}
\noindent{\bf Proof of Claim 4.}\
It is easy to see that $w_1=\frac{3}{5}, w_2=\frac{11}{35}.$ For $3\leqslant i\leqslant n$, expanding $\det W_i$ with respect to its last row, we obtain
$$
w_i=\frac{4}{7}w_{i-1}-\frac{1}{49}w_{i-2}.
$$
Then the characteristic equations of $\{w_i\}_{i\geqslant1}$ is $x^2=\frac{4}{7}x-\frac{1}{49}$, whose roots are $x_1=\frac{2+\sqrt{3}}{7}$ and $x_2=\frac{2-\sqrt{3}}{7}$. Suppose that
\begin{eqnarray}\label{eq:3.5}
w_i=\left(\frac{2+\sqrt{3}}{7}\right)^iz_1+\left(\frac{2-\sqrt{3}}{7}\right)^iz_2,
\end{eqnarray}
then the initial conditions, $w_1=\frac{3}{5}$ and $w_2=\frac{11}{35},$ lead to the system of equations
\begin{eqnarray*}
\left\{
  \begin{array}{l}
    \displaystyle \frac{2+\sqrt{3}}{7}z_1+\frac{2-\sqrt{3}}{7}z_2=\frac{3}{5}, \\[3pt]
    \displaystyle \left(\frac{2+\sqrt{3}}{7}\right)^2z_1+\left(\frac{2-\sqrt{3}}{7}\right)^2z_2=\frac{11}{35}.
  \end{array}
\right.
\end{eqnarray*}
Solving it, we get $z_1=\frac{21+7\sqrt{3}}{30}, z_2=\frac{21-7\sqrt{3}}{30}.$ Thus Claim 4 follows by substituting $z_1$ and $z_2$ back into (\ref{eq:3.5}).\\

By expanding $\det N$ with regards to its last row, we have
$$
\det N=\frac{3}{5}\det W_{n}-\frac{1}{35}\det W_{n-1}=\frac{3}{5}w_{n}-\frac{1}{35}w_{n-1}.
$$
Together with Claim 4, we immediately have the following claim.
\begin{claim}
$
\det N=\frac{49\sqrt{3}}{75}\left[\left(\frac{2+\sqrt{3}}{7}\right)^{n+1}-\left(\frac{2-\sqrt{3}}{7}\right)^{n+1}\right].
$
\end{claim}
The formula of $(-1)^nb_n$ is presented by the following claim. For the sake of convenience, let the diagonal entries of $N$ be $l_{ii}$ and $w_0$ be 1.
\begin{claim}
\begin{align*}
&(-1)^{n}b_{n}\\
=&\frac{343n+441}{150}\left[\left(\frac{2+\sqrt{3}}{7}\right)^{n+1}+\left(\frac{2-\sqrt{3}}{7}\right)^{n+1}\right]
-\frac{21}{50}\left[\left(\frac{2+\sqrt{3}}{7}\right)^n+\left(\frac{2-\sqrt{3}}{7}\right)^n\right]\\
&+\frac{14\sqrt{3}}{225}\left[\left(\frac{2+\sqrt{3}}{7}\right)^n-\left(\frac{2-\sqrt{3}}{7}\right)^n\right].
\end{align*}
\end{claim}
\noindent{\bf Proof of Claim 6.}\
Since $(-1)^{n}b_{n}$ is the sum of those principal minors of $N$ which have $n$ rows and columns, we have
\begin{eqnarray*}
  (-1)^nb_n
            &=&\sum_{i=1}^{n+1}\left|\begin{array}{cccccccc}
                                l_{11} & -\frac{1}{\sqrt{35}} & \cdots & 0 & 0& 0 & 0 & 0 \\
                                -\frac{1}{\sqrt{35}} & l_{22} & \cdots & 0 & 0 & 0 & 0& 0\\
                                \vdots & \vdots & \ddots & \vdots & \vdots & \vdots & \vdots & \vdots \\
                                0 & 0 & \cdots & l_{i-1,i-1} & 0 & 0 & \cdots & 0 \\
                                0 & 0 & \cdots & 0 & l_{i+1,i+1} & -\frac{1}{7} & \cdots & 0 \\
                               \vdots & \vdots & \cdots & \vdots & \vdots & \ddots & \vdots & \vdots \\
                               0 & 0 & \cdots & 0 & 0 & \cdots & l_{n,n}&-\frac{1}{\sqrt{35}}\\
                                0 & 0 & \cdots & 0 & 0 & \cdots & -\frac{1}{\sqrt{35}} &l_{n+1,n+1}
                              \end{array}\right|\\
      &=&\sum_{i=1}^{n+1} \left(\left|\begin{array}{cccc}
                                    l_{11} & -\frac{1}{\sqrt{35}} & \cdots & 0 \\
                                    -\frac{1}{\sqrt{35}} & l_{22} & \cdots & 0 \\
                                    \vdots & \vdots & \ddots & \vdots\\
                                    0 & 0 & \cdots & l_{i-1,i-1}
                                  \end{array}\right|\left|\begin{array}{cccc}
                                                                 l_{i+1,i+1} & -\frac{1}{7} & \cdots & 0 \\
                                                                \vdots & \ddots & \vdots & \vdots  \\
                                                                0 & \cdots & l_{n,n}&-\frac{1}{\sqrt{35}} \\
                                                                 0 & \cdots & -\frac{1}{\sqrt{35}} &l_{n+1,n+1}
                                                               \end{array}
                                                               \right|\right).
\end{eqnarray*}
Note that a permutation similarity transformation of a square matrix preserves its determinant. Together with the property of $N$, we have
$$
 \det W_{n+1-i}=\left|\begin{array}{cccc}
                                                                 l_{i+1,i+1} & -\frac{1}{3} & \cdots & 0 \\
                                                                \vdots & \ddots & \vdots & \vdots  \\
                                                                0 & \cdots & l_{n,n}&-\frac{1}{\sqrt{6}} \\
                                                                 0 & \cdots & -\frac{1}{\sqrt{6}} &l_{n+1,n+1}
                                                               \end{array}
                                                               \right|.
$$
Whence
\begin{eqnarray}\label{eq:3.6}
(-1)^nb_n=\sum_{i=1}^{n+1}w_{i-1}w_{n+1-i}=\sum_{i=0}^nw_iw_{n-i}.
\end{eqnarray}

Together with (\ref{eq:3.6}) and Claim 4, we have
\begin{align*}
(-1)^nb_n=&2w_n+\sum_{i=1}^{n-1}w_iw_{n-i}\\
         =&2w_n+(n-1)\left(\frac{21+7\sqrt{3}}{30}\right)^2\left(\frac{2+\sqrt{3}}{7}\right)^n
         +(n-1)\left(\frac{21-7\sqrt{3}}{30}\right)^2\left(\frac{2-\sqrt{3}}{7}\right)^n\\ &+\frac{49}{150}\sum_{i=1}^{n-1}\left(\frac{2+\sqrt{3}}{7}\right)^i\left(\frac{2-\sqrt{3}}{7}\right)^{n-i}
         +\frac{49}{150}\sum_{i=1}^{n-1}\left(\frac{2-\sqrt{3}}{7}\right)^i\left(\frac{2+\sqrt{3}}{7}\right)^{n-i}\\
         =&2w_n+(n-1)\left(\frac{21+7\sqrt{3}}{30}\right)^2\left(\frac{2+\sqrt{3}}{7}\right)^n
         +(n-1)\left(\frac{21-7\sqrt{3}}{30}\right)^2\left(\frac{2-\sqrt{3}}{7}\right)^n\\ &+\frac{49}{150}\left(\frac{2-\sqrt{3}}{7}\right)^n\sum_{i=1}^{n-1}\left(\frac{2+\sqrt{3}}{2-\sqrt{3}}\right)^i
         +\frac{49}{150}\left(\frac{2+\sqrt{3}}{7}\right)^n\sum_{i=1}^{n-1}\left(\frac{2-\sqrt{3}}{2+\sqrt{3}}\right)^i\\
        =&\frac{343n+441}{150}\left[\left(\frac{2+\sqrt{3}}{7}\right)^{n+1}+\left(\frac{2-\sqrt{3}}{7}\right)^{n+1}\right]
        -\frac{21}{50}\left[\left(\frac{2+\sqrt{3}}{7}\right)^n+\left(\frac{2-\sqrt{3}}{7}\right)^n\right]\\
        &-\frac{14\sqrt{3}}{225}\left[\left(\frac{2+\sqrt{3}}{7}\right)^n-\left(\frac{2-\sqrt{3}}{7}\right)^n\right].
\end{align*}

This completes the proof of Claim 6.

Substituting Claims 5-6 into (\ref{eq:3.4}) yields
$$
\sum_{j=1}^{n+1}\frac{1}{\beta_j}=\frac{\sqrt{3}(7n-9)\left[\left(2+\sqrt{3}\right)^{n+1}+\left(2-\sqrt{3}\right)^{n+1}\right]-4\left[\left(2+\sqrt{3}\right)^{n}-\left(2-\sqrt{3}\right)^{n}\right]}{6\left[\left(2+\sqrt{3}\right)^{n+1}-\left(2-\sqrt{3}\right)^{n+1}\right]}+\frac{9}{2},
$$
as desired.
\end{proof}
By Lemmas 3.1, 3.2 and 3.3, the following result follows immediately.
\begin{thm}
Let $B^2_n$ be the strong prism of a linear polyomino chain with $n$ squares. Then
\begin{align*}
&Kf^*(B^2_n)\\
=&\frac{\sqrt{3}(7n-9)(7n+3)\left[\left(2+\sqrt{3}\right)^{n+1}+\left(2-\sqrt{3}\right)^{n+1}\right]-4(7n+3)\left[\left(2+\sqrt{3}\right)^{n}-\left(2-\sqrt{3}\right)^{n}\right]}{3\left[\left(2+\sqrt{3}\right)^{n+1}-\left(2-\sqrt{3}\right)^{n+1}\right]}\\
&+\frac{49}{3}n^3+70n^2+141n+46.
\end{align*}
\end{thm}

The explicit closed formula of the spanning trees of $B^2_n$ is in the following.
\begin{thm}Let $B^2_n$ be the strong prism of a linear polyomino chain with $n$ squares. Then
$$
\tau(B^2_n)=\sqrt{3}\cdot27\cdot2^{8n-3}\left[\left(2+\sqrt{3}\right)^{n+1}-\left(2-\sqrt{3}\right)^{n+1}\right].
$$
\end{thm}
\begin{proof}
From the proof of Lemma 3.2, we know that $\alpha_i\, (i=2,3,\ldots,n+1)$ is the root of the equation $x^n+a_1x^{n-1}+\cdots+a_{n-1}x+a_n=0.$ Then we have
$$
\prod_{i=2}^{n+1}\alpha_i=(-1)^na_n.
$$
By Claim 2, we have
$$
\prod_{i=2}^{n+1}\alpha_i=\frac{7n+3}{25}\left(\frac{1}{7}\right)^{n-1}.
$$
Similarly,
$$
\prod_{j=1}^{n+1}\beta_j=\frac{49\sqrt{3}}{75}\left[\left(\frac{2+\sqrt{3}}{7}\right)^{n+1}-\left(\frac{2-\sqrt{3}}{7}\right)^{n+1}\right].
$$
Note that
$$
\prod_{i=1}^{4n+4}d_i(B^2_n)=5^8\cdot7^{4n-4},\ \ |E(B^2_n)|=2(7n+3).
$$
Together with Lemma 2.2, we have
\begin{align*}
\tau(B^2_n)&=\frac{1}{2|E(B^2_n)|}\Big[\big(\prod_{i=1}^{4n+4}d_i(B^2_n)\big)\cdot\big(\frac{6}{5}\big)^4\cdot\big(\frac{8}{7}\big)^{2n-2}\cdot\big(\prod_{i=2}^{n+1}2\alpha_i\big)\cdot\big(\prod_{j=1}^{n+1}2\beta_j\big)\Big]\\
&=\sqrt{3}\cdot27\cdot2^{8n-3}\left[\left(2+\sqrt{3}\right)^{n+1}-\left(2-\sqrt{3}\right)^{n+1}\right].
\end{align*}
\end{proof}

At the end of this section, we show that the degree-Kirchhoff index of $B^2_n$ is approximately one eighth of its Gutman index.
\begin{lem}\label{3.6}
Let $B^2_n$ be the strong prism of a linear polyomino chain with $n$ squares. Then
$$
Gut(B^2_n)=\frac{392n^3}{3}+364n^2+\frac{1102n}{3}+38.
$$
\end{lem}
\begin{proof}
By simple calculation, we obtain
\begin{align*}
Gut(B^2_n)&=(n-1)\times7^2\times\dbinom{4}{2}+2\times5^2\times\dbinom{4}{2}+4\times5^2\times(4n+2)+4\times7^2\times\sum_{2\leq i<j\leq n}[4(j-i)+2]\\
&\quad+2\times4\times5\times7\times\sum_{2\leq j\leq n}[4(j-1)+2]\\
&=694n+206+196\sum_{2\leq i<j\leq n}\big[4(j-i)+2\big]+280\sum_{2\leq j\leq n}\big[4(j-1)+2\big]\\
&=694n+206+196\sum_{1\leq i<j\leq n}\big[4(j-i)+2\big]+84\sum_{2\leq j\leq n}\big[4(j-1)+2\big]\\
&=694n+206+196\Big[\frac{2(n+1)n(n-1)}{3}+(n-1)n\Big]+84\big[2(n-1)n+2n-2\big]\\
&=\frac{392n^3}{3}+364n^2+\frac{1102n}{3}+38.
\end{align*}
This completes the proof.
\end{proof}
\begin{thm}
Let $B^2_n$ denote a linear polyomino chain with $n$ squares. Then
$$
\lim_{n\rightarrow \infty}\frac{Kf^*(B^2_n)}{Gut(B^2_n)}=\frac{1}{8}.
$$
\end{thm}
\begin{proof}
By Lemma \ref{3.6}, we have
$$
Gut(B^2_n)=\frac{392n^3}{3}+364n^2+\frac{1102n}{3}+38.
$$
Therefore, together with Theorem 3.4, Theorem 3.7 follows immediately.
\end{proof}

\end{document}